\documentclass[12 pt]{article}
\usepackage[pdftex]{graphicx}
\usepackage{amsmath}
\usepackage{mathtools}
\usepackage{amsthm}
\usepackage{amsfonts}
\usepackage{amssymb}
\usepackage[margin=1.0in]{geometry}
\usepackage{tikz}
\usepackage{enumitem}
\usepackage{cleveref}
\usepackage{verbatim}
\usepackage{url}

\newtheorem{theorem}{Theorem}

\newtheorem{lemma}[theorem]{Lemma}
\newtheorem{prop}[theorem]{Proposition}

\theoremstyle{definition}
\newtheorem{defin}[theorem]{Definition}

\newtheorem{que}[theorem]{Question}

\theoremstyle{remark}
\newtheorem*{rem}{Remark}
\newtheorem*{claim}{Claim}

\renewcommand{\phi}{\varphi}

\newcommand{\aut}{\mathrm{Aut}}

\renewcommand{\cal}[1]{\mathcal{#1}}
\newcommand{\bb}[1]{\mathbb{#1}}

\def\-{\raisebox{.30pt}{-}}

\begin{document}

\title{Minimal subdynamics and minimal flows without characteristic measures}
\author{Joshua Frisch, Brandon Seward, and Andy Zucker}

\date{March 2022}

\maketitle

\begin{abstract}
	Given a countable group $G$ and a $G$-flow $X$, a measure $\mu\in P(X)$ is called characteristic if it is $\aut(X, G)$-invariant. Frisch and Tamuz asked about the existence of a minimal $G$-flow, for any group $G$, which does not admit a characteristic measure. We construct for every countable group $G$ such a minimal flow. Along the way, we are motivated to consider a family of questions we refer to as minimal subdynamics: Given a countable group $G$ and a collection of infinite subgroups $\{\Delta_i: i\in I\}$, when is there a faithful $G$-flow for which every $\Delta_i$ acts minimally?
	\let\thefootnote\relax
	\let\thefootnote\relax\footnote{J.F.\ was supported by NSF Grant DMS-2102838. B.S.\ was supported by NSF Grant DMS-1955090 and Sloan Grant FG-2021-16246. A.Z.\ was supported by NSF Grant DMS-2054302.}
\end{abstract}

Given a countable group $G$ and a faithful $G$-flow $X$, we write $\aut(X, G)$ for the group of homeomorphisms of $X$ which commute with the $G$-action. When $G$ is abelian, $\aut(X, G)$ contains a natural copy of $G$ resulting from the $G$-action, but in general this need not be the case. Much is unknown about how the properties of $X$ restrict the complexity of $\aut(X, G)$; for instance, Cyr and Kra \cite{CK} conjecture that when $G = \bb{Z}$ and $X\subseteq 2^\bb{Z}$ is a minimal, $0$-entropy subshift, then $\aut(X, \bb{Z})$ must be amenable. In fact, no counterexample is known even when restricting to any two of the three properties ``minimal," ``$0$-entropy," or ``subshift." In an effort to shed light on this question, Frisch and Tamuz \cite{FT} define a probability measure $\mu\in P(X)$ to be \emph{characteristic} if it is $\aut(X, G)$-invariant. They show that $0$-entropy subshifts always admit characteristic measures. More recently, Cyr and Kra \cite{CK2} provide several examples of flows which admit characteristic measures for non-trivial reasons, even in cases where $\aut(X, G)$ is non-amenable. Frisch and Tamuz asked (Question 1.5, \cite{FT}) whether there exists, for any countable group $G$, some minimal $G$-flow without a characteristic measure. We give a strong affirmative answer.
\vspace{2 mm}

\begin{theorem}
	\label{Thm:CharMeas}
	For any countably infinite group $G$, there is a free minimal $G$-flow $X$ so that $X$ does not admit a characteristic measure. More precisely, there is a free $(G\times F_2)$-flow $X$ which is minimal as a $G$-flow and with no $F_2$-invariant measure.
\end{theorem}
\vspace{2 mm}

Over the course of proving Theorem~\ref{Thm:CharMeas}, there are two main difficulties to overcome. The first difficulty is a collection of dynamical problems we refer to as \emph{minimal subdynamics}. The general template of these questions is as follows. 
\vspace{2 mm}

\begin{que}
	\label{Que:MinSubdyn}
	Given a countably infinite group $\Gamma$ and a collection $\{\Delta_i: i\in I\}$ of infinite subgroups of $\Gamma$, when is there a faithful (or essentially free, or free) minimal $\Gamma$-flow for which the action of each $\Delta_i$ is also minimal? Is there a natural space of actions in which such flows are generic?
\end{que}
\vspace{2 mm}

In \cite{ZucMinCent}, the author showed that this was possible in the case $\Gamma = G\times H$ and $\Delta = G$ for any countably infinite groups $G$ and $H$. We manage to strengthen this result considerably.
\vspace{2 mm}

\begin{theorem}
	\label{Thm:MinSubDyn}
	For any countably infinite group $\Gamma$ and any collection $\{\Delta_n: n\in \bb{N}\}$ of infinite normal subgroups of $\Gamma$, there is a free $\Gamma$-flow which is minimal as a $\Delta_n$-flow for every $n\in \bb{N}$. 
\end{theorem} 
\vspace{2 mm}

In fact, what we show when proving Theorem~\ref{Thm:MinSubDyn} is considerably stronger. Recall that given a countably infinite group $\Gamma$, a subshift $X\subseteq 2^\Gamma$ is \emph{strongly irreducible} if there is some finite symmetric $D\subseteq \Gamma$ so that whenever $S_0, S_1\subseteq \Gamma$ satisfy $DS_0\cap S_1 = \emptyset$ (i.e.\ $S_0$ and $S_1$ are \emph{$D$-apart}), then for any $x_0, x_1\in X$, there is $y\in X$ with $y|_{S_i} = x_i|_{S_i}$ for each $i< 2$. Write $\cal{S}$ for the set of strongly irreducible subshifts, and write $\overline{\cal{S}}$ for its Vietoris closure. Frisch, Tamuz, and Vahidi-Ferdowsi \cite{FTVF} show that in $\overline{\cal{S}}$, the minimal subshifts form a dense $G_\delta$ subset. In our proof of Theorem~\ref{Thm:MinSubDyn}, we show that the shifts in $\overline{\cal{S}}$ which are $\Delta_n$-minimal for each $n\in \bb{N}$ also form a dense $G_\delta$ subset. 

This brings us to the second main difficulty in the proof of Theorem~\ref{Thm:CharMeas}. Using this stronger form of Theorem~\ref{Thm:MinSubDyn}, one could easily prove Theorem~\ref{Thm:CharMeas} by finding a strongly irreducible $F_2$-subshift which does not admit an invariant measure. This would imply the existence of a strongly irreducible $(G\times F_2)$-subshift without an $F_2$-invariant measure. As not admitting an $F_2$-invariant measure is a Vietoris-open condition, the genericity of $G$-minimal subshifts would then be enough to obtain the desired result. Unfortunately whether such a strongly irreducible subshift can exist (for any non-amenable group) is an open question. To overcome this, we introduce a flexible weakening of the notion of a strongly irreducible shift.

The paper is organized as follows. Section $1$ is a very brief background section on subsets of groups, subshifts, and strong irreducibility. Section $2$ introduces the notion of a UFO, a useful combinatorial gadget for constructing shifts where subgroups act minimally; Theorem~\ref{Thm:MinSubDyn} answers Question 3.6 from \cite{ZucMinCent}. Section $3$ introduces the notion of $\cal{B}$-irreduciblity for any group $H$, where $\cal{B}\subseteq \cal{P}_f(H)$ is a right-invariant collection of finite subsets of $H$. When $H = F_2$, we will be interested in the case when $\cal{B}$ is the collection of finite subsets of $F_2$ which are connected in the standard left Cayley graph. Section 4 gives the proof of Theorem~\ref{Thm:CharMeas}.

\section{Background}

Let $\Gamma$ be a countably infinite group. Given $U, S\subseteq \Gamma$ with $U$ finite, then we call $S$ a (one-sided) \emph{$U$-spaced} set if for every $g\neq h\in S$ we have $h\not\in Ug$, and we call $S$ a \emph{$U$-syndetic} set if $US = \Gamma$. A \emph{maximal $U$-spaced set} is simply a $U$-spaced set which is maximal under inclusion. We remark that if $S$ is a maximal $U$-spaced set, then $S$ is $(U\cup U^{-1})$-syndetic. We say that sets $S, T\subseteq \Gamma$ are (one-sided) \emph{$U$-apart} if $US\cap T = \emptyset$ and $S\cap UT = \emptyset$. Notice that much of this discussion simplifies when $U$ is symmetric, so we will often assume this. Also notice that the properties of being $U$-spaced, maximal $U$-spaced, $U$-syndetic, and $U$-apart are all right invariant.

If $A$ is a finite set or \emph{alphabet}, then $\Gamma$ acts on $A^\Gamma$ by \emph{right shift}, where given $x\in A^\Gamma$ and $g, h\in \Gamma$, we have $(g{\cdot}x)(h) = x(hg)$. A \emph{subshift} of $A^\Gamma$ is a non-empty, closed, $\Gamma$-invariant subset. Let $\mathrm{Sub}(A^\Gamma)$ denote the space of subshifts of $A^\Gamma$ endowed with the Vietoris topology. This topology can be described as follows. Given $X\subseteq A^\Gamma$ and a finite $U\subseteq \Gamma$, the set of \emph{$U$-patterns} of $X$ is the set $P_U(X) = \{x|_U: x\in X\}\subseteq A^U$. Then the typical basic open neighborhood of $X\in \mathrm{Sub}(A^\Gamma)$ is the set $N_U(X):= \{Y\in \mathrm{Sub}(A^\Gamma): P_U(Y) = P_U(X)\}$, where $U$ ranges over finite subsets of $\Gamma$.

A subshift $X\subseteq A^\Gamma$ is \emph{$U$-irreducible} if for any $x_0, x_1\in X$ and any $S_0, S_1\subseteq \Gamma$ which are $U$-apart, there is $y\in X$ with $y|_{S_i} = x_i|_{S_i}$ for each $i< 2$. If $X$ is $U$-irreducible and $V\supseteq U$ is finite, then $X$ is also $V$-irreducible. We call $X$ \emph{strongly irreducible} if there is some finite $U\subseteq \Gamma$ with $X$ $U$-irreducible. By enlarging $U$ if needed, we can always assume $U$ is symmetric. Let $\cal{S}(A^\Gamma)\subseteq \mathrm{Sub}(A^\Gamma)$ denote the set of strongly irreducible subshifts of $A^\Gamma$, and let $\overline{\cal{S}}(A^\Gamma)$ denote the closure of this set in the Vietoris topology.

More generally, if $2^\bb{N}$ denotes Cantor space, then $\Gamma$ acts on $(2^\bb{N})^\Gamma$ by right shift exactly as above, and we will also refer to closed, $\Gamma$-invariant subsets of $(2^\bb{N})^\Gamma$ as subshifts. If $k< \omega$, we let $\pi_k\colon 2^\bb{N}\to 2^k$ denote the restriction to the first $k$ entries. This induces a factor map $\tilde{\pi}_k\colon (2^\bb{N})^\Gamma\to (2^k)^\Gamma$ given by $\tilde{\pi}_k(x)(g) = \pi_k(x(g))$; we also obtain a map $\overline{\pi}_k\colon \mathrm{Sub}((2^\bb{N})^\Gamma)\to \mathrm{Sub}((2^k)^\Gamma)$ (where $2^k$ is viewed as a finite alphabet) given by $\overline{\pi}_k(X) = \tilde{\pi}_k[X]$. The Vietoris topology on $\mathrm{Sub}((2^\bb{N})^\Gamma)$ is the coarsest topology making every such $\overline{\pi}_k$ continuous. We call a subshift $X\subseteq (2^\bb{N})^\Gamma$ \emph{strongly irreducible} if for every $k< \omega$, the subshift $\overline{\pi}_k(X)\subseteq (2^k)^\Gamma$ is strongly irreducible in the ordinary sense. As before, we let $\cal{S}((2^\bb{N})^\Gamma)\subseteq \mathrm{Sub}((2^\bb{N})^\Gamma)$ denote the set of strongly irreducible subshifts of $C^\Gamma$, and we let $\overline{\cal{S}}((2^\bb{N})^\Gamma)$ denote its Vietoris closure.

The idea of considering the closure of the strongly irreducible shifts has it roots in \cite{FT2}. This is made more explicit in \cite{FTVF}, where it is shown that in $\overline{\cal{S}}(A^\Gamma)$, the minimal subflows form a dense $G_\delta$ subset. More or less the same argument shows that the same holds in $\overline{\cal{S}}((2^\bb{N})^\Gamma)$ (see \cite{GTWZ}). Recall that a $\Gamma$-flow $X$ is \emph{free} if for every $g\in \Gamma\setminus \{1_\Gamma\}$ and every $x\in X$, we have $gx\neq x$. The main reason for considering a Cantor space alphabet is the following result, which need not be true for finite alphabets. 
\vspace{2 mm}

\begin{prop}
	\label{Prop:CantorFree}
	In $\overline{\cal{S}}((2^\bb{N})^\Gamma)$, the free subshifts form a dense $G_\delta$ subset. 
\end{prop}

\begin{proof}
	The set $\Omega:= \{(X, x)\in \mathrm{Sub}((2^\bb{N})^\Gamma)\times (2^\bb{N})^\Gamma: x\in X\}$ is closed in the product, and the map $\pi\colon \Omega\to \mathrm{Sub}((2^\bb{N})^\Gamma)$ given by projecting to the first coordinate is open. Fixing $g\in \Gamma\setminus \{1_\Gamma\}$, the set $\{(X, x)\in \Omega: gx\neq x\}$ is open in $\Omega$. It follows that the set of $X\in \mathrm{Sub}((2^\bb{N})^\Gamma)$ for which $gx\neq x$ for every $x\in X$ is also open. Intersecting over all $g\in \Gamma\setminus \{1_\Gamma\}$, we see that freeness is a $G_\delta$ condition.  
	
	Thus it remains to show that freeness is dense in $\overline{\cal{S}}((2^\bb{N})^\Gamma)$. To that end, we fix $g\in \Gamma\setminus \{1_\Gamma\}$ and show that the set of shifts in $\cal{S}((2^\bb{N})^\Gamma)$ where $g$ acts freely is dense. Fix $X\in \cal{S}((2^\bb{N})^\Gamma)$, $k< \omega$, and a finite $U\subseteq \Gamma$; so a typical open set in $\cal{S}((2^\bb{N})^\Gamma)$ has the form $\{X'\in \cal{S}((2^\bb{N})^\Gamma): P_U(\overline{\pi}_k(X')) = P_U(\overline{\pi}_k(X))\}$. We want to produce $Y\in \mathrm{Sub}((2^\bb{N})^\Gamma)$ which is strongly irreducible, $g$-free, and with $P_U(\overline{\pi}_k(Y)) = P_U(\overline{\pi}_k(X))$. In fact, we will produce such a $Y$ with $\overline{\pi}_k(Y) = \overline{\pi}_k(X)$. 
	
	Let $D\subseteq \Gamma$ be a finite symmetric set containing $g$ and $1_\Gamma$. Setting $m = |D|$, consider the subshift $\mathrm{Color}(D, m)\subseteq m^\Gamma$ defined by
	$$\mathrm{Color}(D, m) := \{x\in m^\Gamma: \forall\, i < m\, [x^{-1}(\{i\})\text{ is $D$-spaced}]\}.$$
	A greedy coloring argument shows that $\mathrm{Color}(D, m)$ is non-empty and $D$-irreducible. Moreover, $g$ acts freely on $\mathrm{Color}(D, m)$. Inject $m$ into $2^{\{k,...,\ell-1\}}$ for some $\ell> k$ and identify $\mathrm{Color}(D, m)$ as a subflow of $(2^{\{k,...,\ell-1\}})^\Gamma$. Then $Y:= \overline{\pi}_k(X)\times \mathrm{Color}(D, m)\subseteq (2^\ell)^\Gamma\subseteq (2^\bb{N})^\Gamma$, where the last inclusion can be formed by adding strings of zeros to the end. Then $Y$ is strongly irreducible, $g$-free, and $\overline{\pi}_k(Y) = \overline{\pi}_k(X)$.
\end{proof}

\section{UFOs and minimal subdynamics}

Much of the construction will require us to reason about the product group $G\times F_2$. So for the time being, fix countably infinite groups $\Delta\subseteq \Gamma$. For our purposes, $\Gamma$ will be $G\times F_2$, and $\Delta$ will be $G$, where we identify $G$ with a subgroup of $G\times F_2$ in the obvious way. However, for this subsection, we will reason more generally.
\vspace{2 mm}

\begin{defin}
	\label{Def:UFO}
	Let $\Delta\subseteq \Gamma$ be countably infinite groups. A finite subset $U\subseteq \Gamma$ is called a \emph{$(\Gamma, \Delta)$-UFO} if for any maximal $U$-spaced set $S\subseteq \Gamma$, we have that $S$ meets every right coset of $\Delta$ in $\Gamma$. 
	
	We say that the inclusion of groups $\Delta\subseteq \Gamma$ \emph{admits UFOs} if for every finite $U\subseteq \Gamma$, there is a finite $V\subseteq \Gamma$ with $V\supseteq U$ which is a $(\Gamma, \Delta)$-UFO.
\end{defin}
\vspace{2 mm}

As a word of caution, we note that the property of being a $(\Gamma, \Delta)$-UFO is not upwards closed.

The terminology comes from considering the case of a product group, i.e.\ $\Gamma = \bb{Z}\times \bb{Z}$ and $\Delta = \bb{Z}\times\{0\}$. Figure~\ref{Fig:UFO} depicts a typical UFO subset of $\bb{Z}\times \bb{Z}$.  

\begin{figure}
	\begin{center}
	\begin{tikzpicture}[scale=0.166]
		\foreach \x in {0,1,...,22} {
			\draw [fill=black,radius=0.07] (\x,8) circle;
		}
		\foreach \x in {38,39,...,60} {
			\draw [fill=black,radius=0.07] (\x,8) circle;
		}
		\clip [radius=7.5] (30,8) circle;
		\foreach \x in {23,24,...,37} {
			\foreach \y in {0,1,...,15} {
				\draw [fill=black,radius=0.07] (\x,\y) circle;
			}
		}
	\end{tikzpicture}
	\caption{Sighting in Roswell; a $(\bb{Z}\times \bb{Z}, \bb{Z}\times \{0\})$-UFO subset of $\bb{Z}\times \bb{Z}$.} \label{Fig:UFO}
\end{center}
\end{figure}
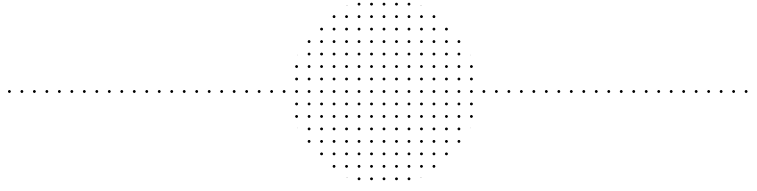


\begin{prop}
	\label{Prop:UFOsExist}
	Let $\Delta$ be a subgroup of $\Gamma$. If $|\bigcap_{u \in U} u \Delta u^{-1}|$ is infinite for every finite set $U \subseteq \Gamma$ then $\Delta \subseteq \Gamma$ admits UFOs. In particular, if $\Delta$ contains an infinite subgroup that is normal in $\Gamma$ then $\Delta\subseteq \Gamma$ admits UFOs.
\end{prop}

\begin{proof}
	We prove the contrapositive. So assume that $\Delta \subseteq \Gamma$ does not admit UFOs. Let $U \subseteq \Gamma$ be a finite symmetric set such that no finite $V \subseteq \Gamma$ containing $U$ is a $(\Gamma, \Delta)$-UFO. Let $D \subseteq \Delta$ be finite, symmetric, and contain the identity. It will suffice to show that $C = \bigcap_{u \in U} u D u^{-1}$ satisfies $|C| \leq |U|$.
	
	Set $V = U \cup D^2$. Since $V$ is not a $(\Gamma, \Delta)$-UFO, there is a maximal $V$-spaced set $S \subseteq \Gamma$ and $g \in \Gamma$ with $S \cap \Delta g = \varnothing$. Since $S$ is $V$-spaced and $u^{-1} C^2 u \subseteq D^2 \subseteq V$, the set $C_u = (u S) \cap (C g)$ is $C^2$-spaced for every $u \in U$. Of course, any $C^2$-spaced subset of $C g$ is empty or a singleton, so $|C_u| \leq 1$ for each $u \in U$. On the other hand, since $S$ is maximal we have $V S = \Gamma$, and since $S \cap \Delta g = \varnothing$ we must have $C g \subseteq U S$. Therefore $|C| = |C g| = \sum_{u \in U} |C_u| \leq |U|$. 
\end{proof}

In the spaces $\overline{\cal{S}}(k^\Gamma)$ and $\overline{\cal{S}}((2^\bb{N})^\Gamma)$, the minimal flows form a dense $G_\delta$. However, when $\Delta\subseteq \Gamma$ is a subgroup, we can ask about the properties of members of $\overline{\cal{S}}(k^\Gamma)$ and $\overline{\cal{S}}((2^\bb{N})^\Gamma)$ viewed as $\Delta$-flows.
\vspace{2 mm}

\begin{defin}
	\label{Def:Min}
	Given a subshift $X\subseteq k^\Gamma$ and a finite $E\subseteq \Gamma$, we say that $X$ is \emph{$(\Delta, E)$-minimal} if for every $x\in X$ and every $p\in P_E(X)$, there is $g\in \Delta$ with $(gx)|_E = p$. Given a subflow $X\subseteq (2^\bb{N})^\Gamma$ and $n\in \bb{N}$, we say that $X$ is \emph{$(\Delta, E, n)$-minimal} if $\overline{\pi}_n(X)\subseteq (2^n)^\Gamma$ is $(\Delta, E)$-minimal. When $\Delta = \Gamma$, we simply say that $X$ is \emph{$E$-minimal} or \emph{$(E, n)$-minimal}.
\end{defin}
\vspace{2 mm}

The set of $(\Delta, E)$-minimal flows is open in $\mathrm{Sub}(k^\Gamma)$, and $X\subseteq k^\Gamma$ is minimal as a $\Delta$-flow iff it is $(\Delta, E)$-minimal for every finite $E\subseteq \Gamma$. Similarly, the set of $(\Delta, E, n)$-minimal flows is open in $\mathrm{Sub}((2^\bb{N})^\Gamma)$, and $X\subseteq (2^\bb{N})^\Gamma$ is minimal as a $\Delta$-flow iff it is $(\Delta, E, n)$ minimal for every finite $E\subseteq \Gamma$ and every $n\in \bb{N}$. 

In the proof of Proposition~\ref{Prop:UFOMinimal}, it will be helpful to extend conventions about the shift action to subsets of $\Gamma$. If $U\subseteq \Gamma$, $g\in G$, and $p\in k^U$, we write $g{\cdot}p\in k^{Ug^{-1}}$ for the function where given $h\in U$, we have $(g{\cdot}p)(h) = p(hg)$. 
\vspace{2 mm}

\begin{prop}
	\label{Prop:UFOMinimal}
	Suppose $\Delta\subseteq \Gamma$ are countably infinite groups and that the inclusion $\Delta\subseteq \Gamma$ admits UFOs. Then the set $\{X\in \overline{\cal{S}}(k^\Gamma): X \text{ is minimal as a $\Delta$-flow}\}$ is a dense $G_\delta$ subset. Similarly, the set $\{X\in \overline{\cal{S}}(2^\bb{N})^\Gamma: X\text{ is minimal as a $\Delta$-flow}\}$ is a dense $G_\delta$ subset.
\end{prop}

\begin{proof}
	 We give the arguments for $k^\Gamma$, as those for $(2^\bb{N})^\Gamma$ are very similar.
	 
	 It suffices to show for a given finite $E\subseteq \Gamma$ that the collection of $(\Delta, E)$-minimal flows is dense in $\overline{\cal{S}}(k^\Gamma)$. By enlarging $E$ if needed, we can assume that $E$ is symmetric. 
	
	Consider a non-empty open $O\subseteq \overline{\cal{S}}(k^\Gamma)$. By shrinking $O$ and/or enlarging $E$ if needed, we can assume that for some $X\in \cal{S}(k^\Gamma)$, we have $O = N_E(X)\cap \overline{\cal{S}}(k^\Gamma)$. We will build a $(\Delta, E)$-minimal shift $Y$ with $Y\in N_E(X)\cap \cal{S}(k^\Gamma)$. Fix a finite symmetric $D\subseteq \Gamma$ so that $X$ is $D$-irreducible. Then fix a finite $U\subseteq \Gamma$ which is large enough to contain an $EDE$-spaced set $Q\subseteq U\cap \Delta$ of cardinality $|P_E(X)|$, and enlarging $U$ if needed, choose such a $Q$ with $EQ\subseteq U$. Fix a bijection $S\to P_E(X)$ by writing $P_E(X) = \{p_g: g\in A\}$. Because $X$ is $D$-irreducible, we can find $\alpha\in P_U(X)$ so that $(gq)|_E = p_g$ for every $g\in Q$. By Proposition~\ref{Prop:UFOsExist}, fix a finite $V\subseteq \Gamma$ with $V\supseteq UDU$ which is a $(\Gamma, \Delta)$-UFO. We now form the shift
	$$Y = \{y\in X: \exists\text{ a max.\ $V$-spaced set $T$ so that }\forall g\in T\, (g{\cdot}y)|_U = \alpha\}.$$
	Because $V = UDU$, we have that $Y\neq \emptyset$. In particular, for any maximal $V$-spaced set $T\subseteq \Gamma$, we can find $y\in Y$ so that $(gy)|_U = \alpha$ for every $g\in T$. We also note that $Y\in N_E(X)$ by our construction of $\alpha$. 
	
	To see that $Y$ is $(\Delta, E)$-minimal, fix $y\in Y$ and $p\in P_E(Y)$. Suppose this is witnessed by the maximal $V$-spaced set $T\subseteq \Gamma$. Because $V$ is a $(\Gamma, \Delta)$-UFO, find $h\in \Delta\cap T$. So $(hy)|_U = \alpha$. Now suppose $g\in Q$ is such that $p = p_g$. We have $(ghy)|_E = (g\cdot((hy)|_U)|_E = p_g$. 
	
	To see that $Y\in \cal{S}(k^\Gamma)$, we will show that $Y$ is $DUVUD$-irreducible. Suppose $y_0, y_1\in Y$ and $S_0, S_1\subseteq \Gamma$ are $DUVUD$-apart. For each $i< 2$, fix $T_i\subseteq \Gamma$ a maximal $V$-spaced set which witnesses that $y_i$ is in $Y$. Set $B_i = \{g\in T_i: DUg\cap S_i\neq \emptyset\}$. Notice that $B_i\subseteq UDS_i$. It follows that $B_0\cup B_1$ is $V$-spaced, so extend to a maximal $V$-spaced set $B$. It also follows that $S_i\cup UB_i\subseteq U^2DS_i$. Since $V\supseteq UDU$ and by the definition of $B_i$, the collection of sets $\{S_i\cup UB_i: i< 2\}\cup \{Ug: g\in B\setminus (B_0\cup B_1)\}$ is pairwise $D$-apart. By the $D$-irreducibility of $X$, we can find $y\in X$ with $y|_{S_i\cup UB_i} = y_i|_{S_i\cup UB_i}$ for each $i< 2$ and with $(gy)|_U = \alpha$ for each $g\in B\setminus (B_0\cup B_1)$. Since $B_i\subseteq T_i$, we actually have $(gy)|_U = \alpha$ for each $g\in B$. So $y\in Y$ and $y|_{S_i} = y_i|_{S_i}$ as desired.
\end{proof}
\vspace{2 mm}

\begin{proof}[Proof of Theorem~\ref{Thm:MinSubDyn}]
	Proposition~\ref{Prop:UFOMinimal} tells us that the generic member of $\overline{\cal{S}}((2^\bb{N})^\Gamma)$ is minimal as a $\Delta_n$-flow for each $n\in \bb{N}$, and by Proposition~\ref{Prop:CantorFree}, the generic member of $\overline{\cal{S}}((2^\bb{N})^\Gamma$ is free.
\end{proof}
\vspace{2 mm}

In contrast to Theorem~\ref{Thm:CharMeas}, the next example shows that Question~\ref{Que:MinSubdyn} is non-trivial to answer in full generality.
\vspace{2 mm}

\begin{theorem}
	Let $G=\sum_\bb{N} (\mathbb{Z}/2\mathbb{Z})$ and let $X$ be a $G$ flow with infinite underlying space. Then there exists an infinite subgroup $H$ such that $X$ is not minimal as an $H$ flow.
\end{theorem}

\begin{proof}
	We may assume that $X$ is a minimal $G$-flow, as otherwise we may take $H = G$. We construct a sequence $X\supsetneq K_0\supseteq K_1\supseteq\cdots$ of proper, non-empty, closed subsets of $X$ and a sequence of group elements $\{g_n: n\in \bb{N}\}$ so that by setting $K = \bigcap_\bb{N} K_n$ and $H = \langle g_n: n\in \bb{N}\rangle$, then $K$ will be a minimal $H$-flow. Start by fixing a closed, proper subset $K_0\subsetneq X$ with non-empty interior. Suppose $K_n$ has been created and is $\langle g_0,...,g_{n-1}\rangle$-invariant. As $X$ is a minimal $G$-flow, the set $S_n:= \{g\in G: \mathrm{Int}(gK_n\cap K_n)\neq \emptyset\}$ is infinite. Pick any $g_n\in S_n\setminus\{1_G\}$, and set $K_{n+1} = g_nK_n\cap K_n$. As $g_n^2 = 1_G$, we see that $K_{n+1}$ is $g_n$-invariant, and as $G$ is abelian, we see that $K_{n+1}$ is also $g_i$-invariant for each $i< n$. It follows that $K$ will be $H$-invariant as desired. 
\end{proof}
\vspace{2 mm}

Before moving on, we give a conditional proof of Theorem~\ref{Thm:CharMeas}, which works as long as some non-amenable group admits a strongly irreducible shift without an invariant measure. It is the inspiration for our overall construction.
\vspace{2 mm}

\begin{prop}
	\label{Prop:HypConstruction}
	Let $G$ and $H$ be countably infinite groups, and suppose that for some $k< \omega$ and some strongly irreducible flow $Y\subseteq k^H$ that $Y$ does not admit an $H$-invariant measure. Then there is a minimal $G$-flow which does not admit a characteristic measure.
\end{prop}

\begin{proof}
	Viewing $Z = k^G\times Y$ as a subshift of $k^{G\times H}$, then $Z$ is strongly irreducible and does not admit an $H$-invariant probability measure. The property of not possessing an $H$-invariant measure is an open condition in $\mathrm{Sub}(k^{G\times H})$. By Proposition~\ref{Prop:UFOMinimal}, we can therefore find $X\subseteq k^{G\times H}$ which is minimal as a $G$-flow and which does not admit an $H$-invariant measure. As $H$ acts by $G$-flow automorphisms on $X$, we see that $X$ does not admit a characteristic measure.
\end{proof}
\vspace{2 mm}

Unfortunately, the question of if there exists any countable group $H$ and a strongly irreducible $H$-subshift $Y$ with no $H$-invariant measure is an open problem. Therefore our construction proceeds by considering the free group $F_2$ and defining a suitable weakening of strongly irreducible subshift which is strong enough for $G$-minimality to be generic in $(G\times F_2)$-subshifts, but weak enough for subshifts without $F_2$-invariant measures to exist.

\section{Variants of strong irreducibility}

In this section, we investigate a weakening of strong irreducibility that one can define given any right-invariant collection $\cal{B}$ of finite subsets of a given countable group. For our overall construction, we will consider $F_2$ and $G\times F_2$, but we give the definitions for any countably infinite group $\Gamma$. Write $\cal{P}_f(\Gamma)$ for the collection of finite subsets of $\Gamma$.
\vspace{2 mm}

\begin{defin}
	\label{Defin:IrreducibleWRTExhaustion}
	Fix a right-invariant subset $\cal{B}\subseteq \cal{P}_f(\Gamma)$. Given $k\in \bb{N}$, we say that a subshift $X\subseteq k^\Gamma$ is \emph{$\cal{B}$-irreducible} if there is a finite $D\subseteq \Gamma$ so that for any $m< \omega$, any $B_0,..., B_{m-1}\in \cal{B}$, and any $x_0,...,x_{m-1}\in X$, if the sets $\{B_0,..., B_{m-1}\}$ are pairwise $D$-apart, then there is $y\in X$ with $y|_{B_i} = x_i|_{B_i}$ for each $i< m$. We call $D$ the \emph{witness} to $\cal{B}$-irreducibility. If we have $D$ in mind, we can say that $X$ is \emph{$\cal{B}$-$D$-irreducible}.
	
	We call a subflow $X\subseteq (2^\bb{N})^\Gamma$ is \emph{$\cal{B}$-irreducible} if for each $k\in \bb{N}$, the subshift $\overline{\pi}_k(X)\subseteq (2^k)^\Gamma$ is $\cal{B}$-irreducible.
	
	We write $\cal{S}_{\cal{B}}(k^\Gamma)$ or $\cal{S}_{\cal{B}}((2^\bb{N})^\Gamma)$ for the set of $\cal{B}$-irreducible subflows of $k^\Gamma$ or $(2^\bb{N})^\Gamma$, respectively, and we write $\overline{\cal{S}}_{\cal{B}}(k^\Gamma)$ or $\overline{\cal{S}}_{\cal{B}}((2^\bb{N})^\Gamma)$ for the Vietoris closures.
\end{defin}
\vspace{2 mm}

\begin{rem}\mbox{}
	\vspace{-3 mm}
	
	\begin{enumerate}
	\item 
	If $\cal{B}$ is closed under unions, it is enough to consider $m = 2$. However, this will often not be the case. 
	\item
	By compactness, if $X\subseteq k^\Gamma$ is $\cal{B}$-$D$-irreducible, $\{B_n: n< \omega\}\subseteq \cal{B}$ is pairwise $D$-apart, and $\{x_n: n< \omega\}\subseteq X$, then there is $y\in X$ with $y|_{B_i} = x_i|_{B_i}$.
	\item
	If $\cal{B}\subseteq \cal{B}'$, then $\cal{S}_{\cal{B}'}(k^\Gamma)\subseteq \cal{S}_{\cal{B}}(k^\Gamma)$ and $\cal{S}_{\cal{B}'}((2^\bb{N})^\Gamma)\subseteq \cal{S}_{\cal{B}}((2^\bb{N})^\Gamma)$
	\end{enumerate}
\end{rem}
\vspace{2 mm}

When $\cal{B}$ is the collection of all finite subsets of $H$, then we recover the notion of a strongly irreducible shift. Again, we consider Cantor space alphabets to obtain freeness.
\vspace{2 mm}

\begin{prop}
	\label{Prop:CantorFreeGeneral}
	For any right-invariant collection $\cal{B}\subseteq \cal{P}_f(\Gamma)$, the generic member of $\overline{\cal{S}}_{\cal{B}}((2^\bb{N})^\Gamma)$ is free.
\end{prop}

\begin{proof}
	Analyzing the proof of Proposition~\ref{Prop:CantorFree}, we see that the only properties that we need of the collections $\cal{S}_{\cal{B}}(k^\Gamma)$ and $\cal{S}_{\cal{B}}((2^\bb{N})^\Gamma)$ for the proof to generalize are that they are closed under products and contain the flows $\mathrm{Color}(D, m)$. If $k, \ell\in \bb{N}$ an $X\subseteq k^\Gamma$ and $Y\subseteq \ell^\Gamma$ are $\cal{B}$-$D$-irreducible and $\cal{B}$-$E$-irreducible for some finite $D, E\subseteq \Gamma$, then $X\times Y\subseteq (k\times \ell)^\Gamma$ will be $\cal{B}$-$(D\cup E)$-irreducible. And as $\mathrm{Color}(D, m)$ is strongly irreducible, it is $\cal{B}$-irreducible.
\end{proof}
\vspace{2 mm}

Now we consider the group $F_2$. We consider the left Cayley graph of $F_2$ with respect to the standard generating set $A:= \{a, b, a^{-1}, b^{-1}\}$. We let $d\colon F_2\times F_2\to \omega$ denote the graph metric. Write $D_n = \{s\in F_2: d(s, 1_{F_2}) \leq n\}$.  
\vspace{2 mm}

\begin{defin}
	\label{Def:ConnSubsets}
 Given $n$ with $1\leq n< \omega$, we set 
	$$\cal{B}_n = \{D\in \cal{P}_f(F_2): \text{ the connected components of $D$ are pairwise $D_n$-apart}\}.$$
	Write $\cal{B}_\omega$ for the collection of finite, connected subsets of $F_2$.
\end{defin}
\vspace{2 mm}

\begin{prop}
	\label{Prop:Reflection}
	Suppose $X\subseteq k^{F_2}$ is $\cal{B}_\omega$-irreducible. Then there is some $n< \omega$ for which $X$ is $\cal{B}_n$-irreducible.
\end{prop}

\begin{proof}
	Suppose $X$ is $\cal{B}_\omega$-$D_n$-irreducible. We show that $X$ is $\cal{B}_n$-$D_n$-irreducible. Suppose $m< \omega$, $B_0,...,B_{m-1}\in \cal{B}_n$ are pairwise $D_n$-apart, and that $x_0,...,x_{m-1}\in X$. For each $i< m$, we suppose $B_i$ has $n_i$-many connected componenets, and we write $\{C_{i,j}: j< n_i\}$ for these components. Then the collection of connected sets $\bigcup_{i< m} \{C_{i,j}: j< n_i\}$ is pairwise $D_n$-apart. As $X$ is $\cal{B}_\omega$-$D_n$-irreducible, we can find $y\in X$ so that for each $i< m$ and $j< n_i$, we have $y|_{C_{i,j}} = x_i|_{C_{i,j}}$. Hence $y|_{B_i} = x_i|_{B_i}$, showing that $X$ is $\cal{B}_n$-$D_n$-irreducible.
\end{proof}
\vspace{2 mm}

We now construct a $\cal{B}_\omega$-irreducible subshift with no $F_2$-invariant measure. We consider the alphabet $A^2$, and write $\pi_0, \pi_1\colon A^2\to A$ for the projections. We set
$$X_{pdox} = \{x\in (A^2)^{F_2}: \forall g, h\in F_2 \, \forall i, j< 2\, [(i, g)\neq (j, h)\Rightarrow \pi_i(x(g))\cdot g\neq \pi_j(x(h))\cdot h]\}.$$
More informally, the flow $X_{pdox}$ is the space of $2$-to-$1$ paradoxical decompositions of $F_2$ where each group element can only move by a generator. This is in some sense the prototypical example of an $F_2$-shift with no $F_2$-invariant measure.

When proving that $X_{pdox}$ is $\cal{B}_\omega$-irreducible, keep in mind that $D_1 = A\cup \{1_{F_2}\}$. 
\vspace{2 mm}

\begin{prop}
	\label{Prop:Paradoxical}
	$X_{pdox}$ is $\cal{B}_\omega$-$D_4$-irreducible.
\end{prop}

\begin{proof}
	Let $B_0,...,B_{k-1}\in \cal{B}_\omega$ be pairwise $D_4$-apart. Let $x_0,...,x_{k-1}\in X_{pdox}$. To construct $y\in X_{pdox}$ with $y|_{B_i} = x_i|_{B_i}$ for each $i< k$, we need to verify a $2$-to-$1$ Hall's matching criterion on every finite subset of $F_2\setminus \bigcup_{i< k} B_i$. Call $s\in F_2$ \emph{matched} if for some $i< k$, some $g\in B_i$, and some $j< 2$, we have $s = \pi_j(x_i(g))\cdot g$. So we need for every finite $D\in \cal{P}_f(F_2\setminus \bigcup_{i<k} B_i)$ that $AD$ contains at least $2|D|$-many unmatched elements. Towards a contradiction, let $D\in \cal{P}_f(F_2\setminus \bigcup_{i<k} B_i)$ be a minimal failure of the Hall condition. 
	
	In the left Cayley graph of $F_2$, given a reduced word $w$ in alphabet $A = \{a, b, a^{-1}, b^{-1}\}$, write $N_w$ for the set of reduced words which \emph{end} with $w$. Now find $t\in D$ (let us assume the leftmost character of $t$ is $a$) so that all of $D\cap N_{a^\frown t}$, $D\cap N_{b^\frown t}$ and $D\cap N_{b^{-1}{}^\frown t}$ are empty. If any two of $a^\frown t$, $b^\frown t$ and $b^{-1}{}^\frown t$ is an unmatched point in the boundary of $D$, then $D\setminus \{t\}$ is a smaller failure of Hall's criterion. So there must be some $i< k$, some $g\in B_i$, and some $j< 2$, we have $\pi_j(x_i(g))\cdot g \in \{a^\frown t, b^\frown t, b^{-1}{}^\frown t\}$. Let us suppose $\pi_j(x_i(g))\cdot g = a^\frown t$. Note that since $g\not\in D$, we must have $g\in \{ba^\frown t, a^2{}^\frown t, b^{-1}a^\frown t\}$.  But then since $B_i$ is connected, we have $D_1B_i\cap \{b^\frown t, b^{-1}{}^\frown t\} = \emptyset$, and since the other $B_q$ are at least distance $5$ from $B_i$, we have $D_1B_q\cap \{b^\frown t, b^{-1}{}^\frown t\} = \emptyset$ for every $q\in k\setminus \{i\}$. In particular, $b^\frown t$ and $b^{-1}{}^\frown t$ are unmatched points in the boundary of $D$, a contradiction.   
\end{proof}
\vspace{2 mm}

We remark that $X_{pdox}$ is not $D_n$-irreducible for any $n\in \bb{N}$. See Figure~\ref{Fig:Xpdox}.

\begin{figure}
	\begin{center}
		\begin{tikzpicture}[scale=1]
			\draw [fill=black,radius=0.08] (4.5*0+2*0,6) circle;
			\node [below] at (4.5*0+2*0+0.5,6.1) {$v_{00}$};
			\draw [-to, thick, red] (4.5*0+2*0-0.2,6) -- (4.5*0+2*0-0.75,6);
			\draw [-to, thick, red] (4.5*0+2*0,5.8) -- (4.5*0+2*0, 5.2);
			\draw [fill=black,radius=0.08] (4.5*0+2*0,5) circle;
			\node [below] at (4.5*0+2*0+0.5,5.1) {$u_{00}$};
			
			\draw [fill=black,radius=0.08] (4.5*0+2*1,6) circle;
			\node [below] at (4.5*0+2*1+0.5,6.1) {$v_{01}$};
			\draw [-to, thick, red] (4.5*0+2*1-0.2,6) -- (4.5*0+2*1-0.75,6);
			\draw [-to, thick, red] (4.5*0+2*1,5.8) -- (4.5*0+2*1, 5.2);
			\draw [fill=black,radius=0.08] (4.5*0+2*1,5) circle;
			\node [below] at (4.5*0+2*1+0.5,5.1) {$u_{01}$};
			
			\draw [thick] (4.5*0+2*0+0.12,4.84) -- (4.5*0+2*0.5-0.12,4.16);
			\draw [thick] (4.5*0+2*1-0.12,4.84) -- (4.5*0+2*0.5+0.12,4.16);
			\draw [fill=black,radius=0.08] (4.5*0+2*0.5,4) circle;
			\node [below] at (4.5*0+2*0.5+0.5,4.1) {$v_0$};
			\draw [-to, dashed, thick, red] (4.5*0+2*0.5-0.2,4) -- (4.5*0+2*0.5-0.75,4);
			\draw [-to, dashed, thick, red] (4.5*0+2*0.5,3.8) -- (4.5*0+2*0.5,3.2);
			\draw [fill=black,radius=0.08] (4.5*0+2*0.5,3) circle;
			\node [below] at (4.5*0+2*0.5+0.5,3.1) {$u_0$};
			
			\draw [fill=black,radius=0.08] (4.5*1+2*0,6) circle;
			\node [below] at (4.5*1+2*0+0.5,6.1) {$v_{10}$};
			\draw [-to, thick, red] (4.5*1+2*0-0.2,6) -- (4.5*1+2*0-0.75,6);
			\draw [-to, thick, red] (4.5*1+2*0,5.8) -- (4.5*1+2*0, 5.2);
			\draw [fill=black,radius=0.08] (4.5*1+2*0,5) circle;
			\node [below] at (4.5*1+2*0+0.5,5.1) {$u_{10}$};
			
			\draw [fill=black,radius=0.08] (4.5*1+2*1,6) circle;
			\node [below] at (4.5*1+2*1+0.5,6.1) {$v_{11}$};
			\draw [-to, thick, red] (4.5*1+2*1-0.2,6) -- (4.5*1+2*1-0.75,6);
			\draw [-to, thick, red] (4.5*1+2*1,5.8) -- (4.5*1+2*1, 5.2);
			\draw [fill=black,radius=0.08] (4.5*1+2*1,5) circle;
			\node [below] at (4.5*1+2*1+0.5,5.1) {$u_{11}$};
			
			\draw [thick] (4.5*1+2*0+0.12,4.84) -- (4.5*1+2*0.5-0.12,4.16);
			\draw [thick] (4.5*1+2*1-0.12,4.84) -- (4.5*1+2*0.5+0.12,4.16);
			\draw [fill=black,radius=0.08] (4.5*1+2*0.5,4) circle;
			\node [below] at (4.5*1+2*0.5+0.5,4.1) {$v_1$};
			\draw [-to, dashed, thick, red] (4.5*1+2*0.5-0.2,4) -- (4.5*1+2*0.5-0.75,4);
			\draw [-to, dashed, thick, red] (4.5*1+2*0.5,3.8) -- (4.5*1+2*0.5,3.2);
			\draw [fill=black,radius=0.08] (4.5*1+2*0.5,3) circle;
			\node [below] at (4.5*1+2*0.5+0.5,3.1) {$u_1$};
			
			\draw [thick] (4.5*0+2*0.5+0.16,3-0.12) -- (3.25-0.16,1.5+0.12);
			\draw [thick] (4.5*1+2*0.5-0.16,3-0.12) -- (3.25+0.16,1.5+0.12);
			\draw [fill=black,radius=0.08] (3.25,1.5) circle;
			\node [below] at (3.25+0.5,1.6) {$v_\varnothing$};
			\draw [-to, dashed, thick, red] (3.25-0.2,1.5) -- (3.25-0.75,1.5);
			\draw [-to, dashed, thick, red] (3.25,1.5-0.2) --(3.25,0.2);
			\draw [fill=black,radius=0.08] (3.25,0) circle;
			\node [below] at (3.25+0.5,0.1) {$u_\varnothing$};
			]
		\end{tikzpicture}
		\caption{A pair of outgoing edges, drawn in solid red, is chosen at each of $v_{00}$, $v_{01}$, $v_{10}$, and $v_{11}$. Edges which must consequently be oriented in a particular direction are indicated with dashed red arrows. Most importantly, $v_{\varnothing}$ is forced to direct an edge to $u_\varnothing$. By considering the generalization of this picture for any length of binary string, we see that $X_{pdox}$ cannot be $D_n$-irreducible for any $n\in \bb{N}$.
		} \label{Fig:Xpdox}
	\end{center}
\end{figure}
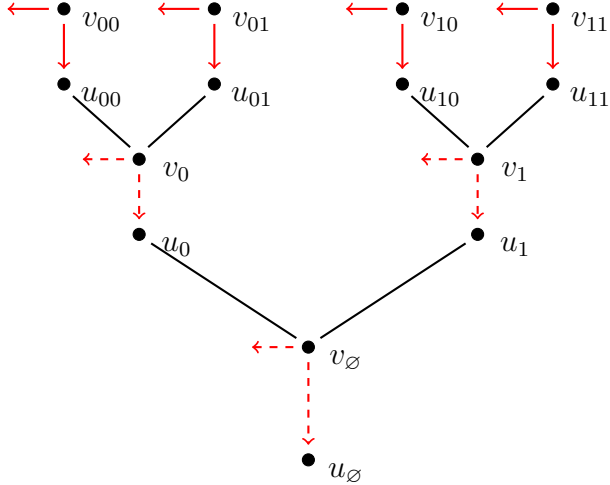

\section{The construction}

Our goal for the rest of the paper is to use $X_{pdox}$ to build a subshift of $(2^\bb{N}))^{G\times F_2}$ which is free, $G$-minimal, and with no $F_2$-invariant measure. In what follows, given an $F_2$-coset $\{g\}\times F_2$, we endow this coset with the left Cayley graph for $F_2$ using the generating set $A$ exactly as above. We extend the definition of $\cal{B}_n$ to refer to finite subsets of any given $F_2$-coset. 
\vspace{2 mm}

\begin{defin}
	\label{Def:ConnSubsetsProduct}
	Given $n$ with $1\leq n\leq \omega$, we set 
	$$\cal{B}_n^*= \{D\in \cal{P}_f(G\times F_2): \text{ for each $F_2$-coset $C$, $D\cap C\in \cal{B}_n$}\}.$$ 
\end{defin}
\vspace{2 mm}

Given $y\in k^{G\times F_2}$ and $g\in G$, we define $y_g\in k^{F_2}$ where given $s\in F_2$, we set $y_g(s) = y(g, s)$. If $X\subseteq k^{F_2}$ is $\cal{B}_n$-irreducible, then the subshift $X^G\subseteq k^{G\times F_2}$ is in $\cal{S}_n$, where we view $X^G$ as the set $\{y\in k^{G\times F_2}: \forall g\in G\, (y_g\in X)\}$. In particular, $(X_{pdox})^G$ is $\cal{B}^*_4$-irreducible. By encoding $X_{pdox}$ as a subshift of $(2^m)^{G\times F_2}$ for some $m\in \bb{N}$ and considering $\tilde{\pi}_m^{-1}(X_{pdox})\subseteq (2^\bb{N})^{G\times F_2}$, we see that there is a $\cal{B}_4^*$-irreducible subflow of $(2^\bb{N})^{G\times F_2}$ for which the $F_2$-action doesn't fix a measure. It follows that such subflows constitute a non-empty open subset of $\Phi:= \overline{\bigcup_n \cal{S}_{\cal{B}_n^*}((2^\bb{N})^{G\times F_2})}$. Combining the next result with Proposition~\ref{Prop:CantorFreeGeneral}, we will complete the proof of Theorem~\ref{Thm:CharMeas}.
\vspace{2 mm}

\begin{prop}
	\label{Prop:ProductGenericity}
	With $\Phi$ as above, the $G$-minimal flows are dense $G_\delta$ in $\Phi$. 
\end{prop}

\begin{proof}
	We show the result for $\Phi_k:= \overline{\bigcup_n \cal{S}_{\cal{B}_n^*}(k^{G\times F_2})}$ to simplify notation; the proof in full generality is almost identical.

	We only need to show density. To that end, fix a finite symmetric $E\subseteq G\times F_2$ which is connected in each $F_2$-coset. It is enough to show that the $(G, E)$-minimal subshifts are dense in $\Phi_k$. Fix some non-empty open $O\subseteq \Phi_k$. By enlarging $E$ and/or shrinking $O$, we may assume that for some $n< \omega$ and $X\in \cal{S}_{\cal{B}_n^*}(k^{G\times F_2})$ that $O = \{X'\in \Phi_k: P_E(X') = P_E(X)\}$. We will build a $(G, E)$-minimal subshift $Y$ so that $P_E(Y) = P_E(X)$ and so that for some $N< \omega$, we have $Y\in \cal{S}_{\cal{B}_N^*}(k^{G\times F_2})$. 
	
	Recall that $D_n\subseteq F_2$ denotes the ball of radius $n$. Fix a finite, symmetric $D\subseteq G\times F_2$ so that $D_{2n}\subseteq D$ and $X$ is $\cal{B}_n^*$-$D$-irreducible. Find a finite symmetric $U_0\subseteq G$ with $1_G\subseteq U_0$ and $r< \omega$ so that upon setting $U = U_0\times D_r\subseteq G\times F_2$, then $U$ is large enough to contain an $EDE$-spaced set $Q\subseteq G$ with $EQ\subseteq U$. As $X$ is $\cal{B}_n^*$-$D$-irreducible, there is a pattern $\alpha\in P_U(X)$ so that $\{(g\alpha)|_E: g\in Q\} = P_E(X)$. 
	
	Let $V\supseteq UD^2U$ be a $(G\times F_2, G)$-UFO. We remark that for most of the remainder of the proof, it would be enough to have $V\supseteq UDU$; we only use the stronger assumption $V\supseteq UD^2U$ in the proof of the final claim. Consider the following subshift:
	$$Y = \{y\in X: \exists\text{ a max.\ $V$-spaced set $T$ so that }\forall g\in T\, (gy)|_U = \alpha\}.$$
	The proof that $Y$ is non-empty and $(G, E)$-minimal is exactly the same as the analogous proof from Proposition~\ref{Prop:UFOMinimal}.
	
	We now show that $Y\in \cal{S}_{\cal{B}_N^*}(k^{G\times F_2})$ for $N = 4r+3n$. Set $W = DUVUD$. We show that $Y$ is $\cal{B}_N^*$-$W$-irreducible. Suppose $m< \omega$, $y_0,...,y_{m-1}\in Y$ and $S_0,...,S_{m-1}\in \cal{B}_N^*$ are pairwise $W$-apart. Suppose for each $i< m$ that $T_i\subseteq G\times F_2$ is a maximal $V$-spaced set which witness that $y_i\in Y$. Set $B_i = \{g\in T_i: DUg\cap S_i\neq \emptyset\}$. Then $\bigcup_{i< m} B_i$ is $V$-spaced, so enlarge to a maximal $V$-spaced set $B\subseteq G\times F_2$. 
	
	For each $i< m$, we enlarge $S_i\cup UB_i$ to $J_i\in\cal{B}_n^*$ as follows. Suppose $C\subseteq G\times F_2$ is an $F_2$-coset. Each set of the form $C\cap Ug$ is connected. Since $S_i\in \cal{B}_N^*$, it follows that given $g\in B_i$, there is at most one connected component $\Theta_{C, g}$ of $S_i\cap C$ with $Ug\cap \Theta_g = \emptyset$, but $Ug\cap D_n\Theta_g\neq \emptyset$. We add the line segment in $C$ connecting $\Theta_{C, g}$ and $Ug$. Upon doing this for each $g\in B_i$ and each $F_2$-coset $C$, this completes the construction of $J_i$. Observe that $J_i\subseteq D_{n-1}S_i\cap UB_i$.
	\vspace{2 mm}
	
	\begin{claim}
		Let $C$ be an $F_2$-coset, and suppose $Y_0$ is a connected component of $S_i\cap C$. Let $Y$ be the connected component of $J_i\cap C$ with $Y_0\subseteq Y$. Then $Y\subseteq D_{2r+n}Y_0$. In particular, if $Y_0\neq Z_0$ are two connected components of $S_i\cap C$, then $Y_0$ and $Z_0$ do not belong to the same component of $J_i\cap C$.  
	\end{claim}

	\begin{proof}
		Let $L = \{x_j: j< \omega\}\subseteq C$ be a ray with $x_0\in Y_0$ and $x_j\not\in Y_0$ for any $j\geq 1$. Then $\{j< \omega: x_j\in J_i\cap C\}$ is some finite initial segment of $\omega$. We want to argue that for some $j\leq 2r+n+1$, we have $x_j\not\in J_i\cap C$. First we argue that if $x_n\in J_i\cap C$, then $x_n\in UB_i$. Otherwise, we must have $x_n\in D_{n-1}S_i$. But since $x_n\not\in D_{n-1}Y_0$, there must be another component $Y_1$ of $S_i\cap C$ with $x_n\in D_nY_1$. But this implies that $Y_0$ and $Y_1$ are not $D_{2n-1}$-apart, a contradiction since $2n-1\leq 4r-3n = N$.
		
		Fix $g\in B_i$ with $x_n\in Ug$. Let $q< \omega$ be least with $q> n$ and $x_q \not\in U_g$. We must have $q\leq 2r+n+1$. We claim that $x_q\not\in J_i\cap C$. Towards a contradiction, suppose $x_q\in J_i\cap C$. We cannot have $x_q\in UB_i$, so we must have $x_q\in D_{n-1}S_i$. But now there must be some component $Y_1$ of $S_i\cap C$ with $x_q\in D_{n-1}Y_1$. But then $D_{2r+2n}Y_0\cap Y_1\neq \emptyset$, a contradiction as $Y_0$ and $Y_1$ are $D_N$-apart. This concludes the proof that $Y\subseteq D_{2r+n}Y_0$.
		
		Now suppose $Y_0\neq Z_0$ are two connected components of $S_i\cap C$. Then $Y_0$ and $Z_0$ are $N$-apart. In particular, $Z_0\not\subseteq D_{2r+n}Y_0$, so cannot belong to the same connected component of $J_i\cap C$ as $Y_0$.  
	\end{proof}
		
	\begin{claim}
		$J_i\in \cal{B}_n^*$. 	
	\end{claim}	
		
	\begin{proof}	
		Fix an $F_2$-coset $C$ and two connected components $Y\neq Z$ of $J_i\cap C$. By the previous claim, each of $Y$ and $Z$ can only contain at most one non-empty component of $S_i\cap C$. The claim will be proven after considering three cases. 
		\begin{enumerate}
			\item 
			First suppose each of $Y$ and $Z$ contain a non-empty component of $S_i\cap C$, say $Y_0\subseteq Y$ and $Z_0\subseteq Z$. Then since $Y_0$ and $Z_0$ are $D_{4r+3n}$-apart, the previous claim implies that $Y$ and $Z$ are $D_n$-apart.
			\item
			Now suppose $Y$ contains a non-empty component $Y_0$ of $S_i\cap C$ and that $Z$ does not. Then for some $g\in B_i$, we have $Z = Ug\cap C$. Towards a contradiction, suppose $D_nY\cap Ug \neq \emptyset$. Let $L = \{x_j: j\leq M\}$ be the line segment connecting $Y$ and $Ug$ with $L\cap Y = \{x_0\}$ and $L\cap Ug = \{x_M\}$. We must have $M\leq n$. We cannot have $x_0\in UB_i$, so we must have $x_0\in D_{n-1}S_i$. This implies that $x_0\in D_{n-1}Y_0$. We cannot have $x_0\in Y_0$, as otherwise, we would have connected $Y_0$ and $Ug\cap C$ when constructing $J_i$. It follows that for some $h\in B_i$, we have that $x_0$ is on the line segment $L' = \{x_j': j\leq  M'\}$ connecting $Y_0$ and $Uh\cap C$, and we have $M'\leq n$. But this implies that $Ug\cap D_{2n}Uh\neq \emptyset$, a contradiction since $V\supseteq UDU$ and $D\supseteq D_{2n}$.
			\item
			If neither $Y$ nor $Z$ contain a component of $S_i\cap C$, then there are $g\neq h\in B_i$ with $Y = Uh\cap C$ and $Z = Ug\cap C$. It follows that $Y$ and $Z$ are $D_n$-apart.\qedhere
		\end{enumerate} 
	\end{proof}
	
	\begin{claim}
		Suppose $i\neq j< m$. Then $J_i$ and $J_j$ are $D$-apart.
	\end{claim} 

	\begin{proof}
		We have that $J_i\subseteq D_{n-1}S_i\cup UB_i$, and likewise for $j$. As $UB_i\subseteq U^2DS_i$ and as $D\supseteq D_{2n}$, we have $J_i\subseteq U^2DS_i$, and likewise for $j$. As $S_i$ and $S_j$ are $W$-apart and as $V\supseteq UDU$, we see that $J_i$ and $J_j$ are $D$-apart.
	\end{proof}

	\begin{claim}
		Suppose $g\in B\setminus \bigcup_{i< m} B_i$. Then $Ug$ and $J_i$ are $D$-apart for any $i< m$.
	\end{claim}

	\begin{proof}
		As $g\not\in B_i$, we have $U_g$ and $S_i$ are $D$-apart. Also, for any $h\in B$ with $g\neq h$, we have that $Ug$ and $Uh$ are $D$-apart. Now suppose $DUg\cap J_i\neq \emptyset$. If $x\in DUg\cap J_i$, then on the coset $C = F_2x$, $x$ must belong on the line between a component of $S_i\cap C$ and $Uh$ for some $h\in B_i$. Furthermore, we have $x\in D_{n-1}Uh$. But since $D_{2n}\subseteq D$, this contradicts that $Ug$ and $Uh$ are $D^2$-apart (using the full assumption $V\supseteq UD^2U$). 
	\end{proof}

	We can now finish the proof of Proposition~\ref{Prop:ProductGenericity}. The collection $\{J_i: i< m\}\cup \{Ug: g\in B\setminus (\bigcup_{i< m} B_i)\}$ is a pairwise $D$-apart collection of members of $\cal{B}_n^*$. As $X$ is $\cal{B}_n$-$D$-irreducible, we can find $y\in X$ with $y|_{J_i} = y_i|_{J_i}$ for each $i< m$ and with $(gy)|_U = \alpha$ for each $g\in B\setminus (\bigcup_{i< m} B_i)$. As $J_i\supseteq UB_i$ and since $B_i\subseteq T_i$, we actually have $(gy)|_U = \alpha$ for each $g\in B$. As $B$ is a maximal $V$-spaced set, it follows that $y\in Y$ and $y|_{S_i} = y_i|_{S_i}$ as desired.
\end{proof}


\begin{thebibliography}{99}
	\bibitem{CK} V.\ Cyr and B. Kra, The automorphism group of a minimal shift of stretched exponential
	growth, \emph{Journal of Modern Dynamics}, \textbf{10} (2016), 483--495.
	
	\bibitem{CK2} V.\ Cyr and B.\ Kra, Characteristic measures for language stable subshifts, \emph{submitted}, \url{https://arxiv.org/abs/2101.12669}.
	
	\bibitem{FT} J.\ Frisch and O.\ Tamuz, Characteristic measures of symbolic dynamical systems, \emph{Ergodic Theory Dyn.\ Sys.}, to appear. 
	
	\bibitem{FT2} J.\ Frisch and O.\ Tamuz, Symbolic dynamics on amenable groups: the entropy of generic shifts, \emph{Ergodic Theory Dyn.\ Sys.}, \textbf{37(4)}, 1187--1210.
	
	\bibitem{FTVF} J.\ Frisch, O.\ Tamuz, and P.\ Vahidi Ferdowsi, Strong amenability and the infinite conjugacy class property, \emph{Inventiones Mathematicae}, \textbf{218} (2019), 833-851.
	
	\bibitem{GTWZ}
	E. Glasner, T. Tsankov, B. Weiss, and A. Zucker, Bernoulli disjointness, \emph{Duke Mathematical Journal}, \textbf{170(4)} (2021), 615--651.
	
	\bibitem{ZucMinCent} A.\ Zucker, Minimal flows with arbitrary centralizer, \emph{Ergodic Theory Dyn.\ Sys.}, (2020), DOI: 10.1017/etds.2020.128.
\end{thebibliography}
\end{document}